\documentclass[10pt]{article}
\usepackage{amsmath,amssymb,amsthm}

\theoremstyle{plain}
\newtheorem{theorem}{Theorem}
\newtheorem{lemma}[theorem]{Lemma}

\newtheorem{corollary}[theorem]{Corollary}
\newtheorem{claim}{Claim}

\theoremstyle{definition}
\newtheorem{definition}[theorem]{Definition}
\newtheorem{question}[theorem]{Question}

\numberwithin{theorem}{section}

\title{\bf On partition and almost disjoint properties of combinatorial notions}
\date{}
\author{Teng Zhang
        \footnote{School of Science, Zhejiang University of Science and
Technology, Liuhe Road, Hangzhou, 310023, Zhejiang, China\hfill\break
{\tt zhangteng@zust.edu.cn}\hfill\break
{Keywords: infinite partition, almost disjoint family, the Central Set Theorem, Ramsey Theory \hfill\break
}}
}

\begin{document}
\maketitle

\begin{abstract}
  It is known that there are many notions of largeness in a semigroup that own rich combinatorial properties. In this paper, we focus on partition and almost disjoint properties of these notions. One of the most remarkable results with respect to this topic is that in an infinite very weakly cancellative semigroup of size $\kappa$, every central set can be split into $\kappa$ disjoint central subsets. Moreover, if $\kappa$ contains $\lambda$ almost disjoint subsets, then every central set contains a family of $\lambda$ almost disjoint central subsets. And many other combinatorial notions are found successively to have analogous properties, among these are thick sets, piecewise syndetic sets, $J$-sets and $C$-sets. In this paper, we mainly study four other notions: IP sets, combinatorially rich sets, $C_p$-sets and PP-rich sets. Where the latter two are known in $(\mathbb{N}, +)$, related to the polynomial extension of the central sets theorem. We lift them up to commutative cancellative semigroups and obtain an uncountable version of the polynomial extension of the central sets theorem incidentally. And we finally find that the infinite partition and almost disjoint properties hold for $C_p$-sets in commutative cancellative semigroups and for other three notions in $(\mathbb{N}, +)$.  
\end{abstract}

\section{Introduction}
In Ramsey Theory, there are many notions of largeness in a semigroup that originated in topological dynamics, such as IP sets, central sets, piecewise syndetic sets and so on. They are found to own rich combinatorial properties. Actually many combinatorial or arithmetic results attribute to specific properties of certain notions. And most of them have algebraic descriptions in the Stone–\v{C}ech compactification of semigroups, which have already become a powerful tool to study their combinatorial properties further. In this article, we investigate one of important combinatorial properties - infinite partition and almost disjoint properties. 

With respect to this study, the notion of central sets is a successful case. This notion was introduced by Furstenburg \cite{1981Recurrence} in topological dynamics when he studied finite systems of equations satisfying Rado's columns condition. Furstenburg obtained many properties of central sets, especially the famous central set theorem \cite[Proposition 8.21]{1981Recurrence}, so that this notion gradually gained attention from other mathematicians. In \cite{1990Nonmetrizable} and \cite{1996Nonmetrizable}, authors established an equivalent characterization in terms of the algebraic structure of the Stone– \v{C}ech compactification of discrete semigroups. This characterization directly implies that for any 2-partition of a central set in any semigroup, there must exist one cell which is still central. Then a natural question arises: whether a central set can be split into two disjoint central subsets? In the case of $(\mathbb{N}, +)$, \cite[Theorem 2.12]{2003Infinite} gave a positive answer. Hence we can immediately obtain that any central set in $\mathbb{N}$ can be split into infinite many central subsets. From this conclusion, a series of further questions appear, where the following three are worth to notice:
\begin{enumerate}
  \item How many almost disjoint central subsets does a given central set contain at most?
  \item Do infinite partition or almost disjoint properties of central sets still hold if the semigroup is uncountable?
  \item What about other combinatorial notions?
\end{enumerate}
Where if $X$ is an infinite set, we call $\mathcal{A}$ is a family of almost disjoint subsets of $X$ if and only if for each $A \in \mathcal{A}$, $A \subseteq X$ and $|A| = |X|$, and for any distinct $A, B \in \mathcal{A}$, $|A \cap B| < |X|$. This is a basic but important concept in combinatorial set theory, which often shows up with partition problems.
In \cite{2008Almost}, authors mainly investigated the above three questions and obtained several results. Where the most notable of these is that \cite[Corollary 3.4]{2008Almost} in any infinite very weakly cancellative semigroup of size $\kappa$, every central set can be split into $\kappa$ pairwise disjoint central subsets; and if $\kappa$ contains $\lambda$ almost disjoint subsets, then every central set contains $\lambda$ almost disjoint central subsets. This result answers the first two questions. With respect to the third question, authors\cite{2008Almost} studied several important notions: thick sets, very thick sets, piecewise syndetic sets and syndetic sets. They found that the statement for thick sets is the same as that for central sets; if the size of the semigroup is regular, then the statement holds for very thick sets; if the semigroup is left cancellative, then the statement holds for piecewise syndetic sets. While the situation of syndetic sets is more complicated, see \cite[Section 4]{2008Almost} for more details. In \cite{2024Partition}, the author investigated three other notions: quasi-central sets, $J$-sets and $C$-sets. Where the definition of the first one is similar to that of central sets, and the latter two notions are related to a stronger version of the central set theorem \cite{2008A}. And the author showed that quasi-central sets have the same statement as central sets; if the semigroup is commutative, then the statement holds for $C$-sets, and if the size of the semigroup $\kappa$ satisfies $\kappa^\omega = \kappa$, then the statement also holds for $J$-sets. Besides that, some mathematicians considered similar partition problems in groups. For example, in \cite[Chapter 3]{2003Ball}, authors showed that every infinite group can be partitioned into infinitely many sets that are both left and right syndetic. In \cite{2011Partitions, 2012Prethick}, authors studied the partition of groups into $\kappa$-thin subsets and not $k$-prethick subsets respectively. There are still many other studies with respect to this topic, such as \cite{2015Partitions}, we do not go into details.

In this article, we will continue this topic. To be precise, we mainly investigate four combinatorial notions: IP sets, combinatorially rich sets, $C_p$-sets and PP-rich sets. The first notion is related to the famous Hindman theorem\cite[Theorem 3.1]{1974Finite}, defined as follows:
\begin{definition}
  Suppose $(S, +)$ is a semigroup and $A \subseteq S$. $A$ is an IP set if there is a sequence $\langle x_n \rangle_{n=1}^\infty$ in $S$ such that $\mathrm{FS}(\langle x_n \rangle_{n=1}^\infty) \subseteq A$. 
\end{definition}
Where $\mathrm{FS}(\langle x_n \rangle_{n=1}^\infty) = \{\sum_{n \in H}x_n: H \in \mathcal{P}_f(\mathbb{N}) \}$, $\mathcal{P}_f(\mathbb{N})$ is the set of nonempty finite subsets of $\mathbb{N}$ and $\sum_{n \in H}x_n$ is the sum in increasing order of indices. From the definition we can see that IP sets contain rich additive structures. Actually this notion has an equivalent algebraic characterization: $A$ is an IP set in $S$ if and only if there exists an idempotent $p$ in $\beta S$ such that $A \in p$, where $\beta S$ is the Stone– \v{C}ech compactification of $S$ we will introduce below.

In Section 2, we shall discuss the infinite partition and almost disjoint properties of IP sets. Observe that the cardinalities of IP sets in uncountable semigroups are not unique, so we will study from two aspects. Firstly, we consider which properties hold for any IP set. And we will see that in any left weakly cancellative semigroup, every IP set contains $\omega$ pairwise disjoint IP subsets and $2^\omega$ almost disjoint IP subsets. So these conclusions are already optimal for countable IP sets. Secondly and naturally, when the IP set is known to be uncountable, we show that uncountable IP sets still have almost disjoint properties (Theorem \ref{IPADU}). However, we find that uncountable IP sets can not always be split into uncountably many IP cells. We provide a necessary and sufficient condition (Theorem \ref{NS}) and corresponding examples (e.g. Theorem \ref{EX}). 

The second notion was introduced by Bergelson and Glasscock\cite[Defninition 2.8]{2020On} in commutative semigroups. To show the definition of combinatorially rich sets more succinctly, let us introduce some notations first. For a set $X$ and a natural number $n \in \mathbb{N}$, we denote $\mathcal{P}_n(X)$ as the set of all size $n$ subsets of $X$ and denote ${^nX}$ as the set of all sequences of $X$ of length $n$. If $(S, +)$ is a semigroup, then for $n, m \in \mathbb{N}$, $L \in \mathcal{P}_n({^mS})$, $a \in S$ and nonempty $H \subseteq \{ 1, \ldots, m \}$, we denote $S_L(a, H) = \{ a + \sum_{t \in H}f(t): f \in L \}$.
\begin{definition}
  Suppose $(S, +)$ is a commutative semigroup and $A \subseteq S$. If there exists a sequence $\langle r_n \rangle_{n=1}^\infty$ in $\mathbb{N}$ such that for each $n \in \mathbb{N}$ and each $L \in \mathcal{P}_n({^{r_n}S})$, there exist $a \in S$ and nonempty $H \subseteq \{ 1, \ldots, r_n \}$ such that $S_L(a, H) \subseteq A$, then we say $A$ is a combinatorially rich set in $S$.
\end{definition}

This notion has already been shown in \cite{2020On} to contain an abundance of combinatorial patterns, and which was lifted up to arbitrary semigroups by Hindman et al.\cite{2023Combinatorially} recently. This notion has already known to have partition regularity \cite[Theorem 2.4]{2023Combinatorially}, that is, any 2-partition of a combinatorially rich set must contain one combinatorially rich cell. In Section 3, we will study its infinite partition and almost disjoint properties. And we finally obtain that in $(\mathbb{N}, +)$, any combinatorially rich set contains $\omega$ pairwise disjoint combinatorially rich subsets and $2^\omega$ almost disjoint combinatorially rich subsets. Unfortunately, we do not know whether analogous results hold in uncountable semigroups, so we list it as a question for future study. 

The latter two notions are studied in \cite{2023Polynomial}, where authors lifted $J$-sets and $C$-sets up to polynomial versions (called $J_p$-sets and $C_p$-sets) in $\mathbb{N}$, and established a polynomial version of the central set theorem for $(\mathbb{N}, +)$. PP-rich sets are a kind of notion related to $J_p$-sets, and are proved to have partition regularity\cite[Theorem 19]{2023Polynomial}. However, all these notions and relevant results, especially the polynomial version of the central set theorem, are in $(\mathbb{N}, +)$. Although authors in \cite{2023Polynomial} noted that most of them can be lifted up to the case of countable commutative semigroups, we still do not know whether there are uncountable versions of these combinatorial notions and corresponding results. 

So in Section 4, we shall extend the definitions of $J_p$-sets and $C_p$-sets to commutative cancellative semigroups (Definition \ref{defJC}). And then establish an uncountable polynomial version of the central set theorem (Theorem \ref{UPV}). Furthermore in the last section, we will show the infinite partition and almost disjoint results of $C_p$-sets in commutative cancellative semigroups (Theorem \ref{Cp}). For PP-rich sets (Definition \ref{PPR}), we obtain Theorem \ref{PPN} that in $(\mathbb{N}, +)$, any PP-rich set contains $\omega$ pairwise disjoint PP-rich subsets and $2^\omega$ almost disjoint PP-rich subsets. We also leave several questions in the end of the paper. We do not know the uncountable situation of PP-rich sets like combinatorially rich sets. The most tough notion to deal with is $J_p$-sets, we do not obtain any corresponding partition or almost disjoint results, even in $(\mathbb{N}, +)$. And the partition regularity of $J_p$-sets is also unknown, which is an open question in \cite{2023Polynomial} when the semigroup is $(\mathbb{N}, +)$ (in this situation, we obtain a partial answer Theorem \ref{partialA} that if $A$ is a $J_p$-set in $\mathbb{N}$ and $B$ is a finite subset of $\mathbb{N}$, then $A \setminus B$ is also a $J_p$-set). All these questions will be studied further in the future.

Now let us introduce some notions, notations and basic facts that we will refer to. Most of this information can be found in \cite{2012Algebra}. Given a discrete semigroup $(S, \cdot)$, $\beta S$ is the Stone-\v{C}ech compactification of $S$ and there is a natural extension of $\cdot$ to $\beta S$ making $\beta S$ a compact right topological semigroup.
For each $p \in \beta S$, the function $\rho_p: \beta S \rightarrow \beta S$, defined by $\rho_p(q) = q \cdot p$, is continuous, and for each $x \in S$,
$\lambda_x: \beta S \rightarrow \beta S$, defined by $\lambda_x(p) = x \cdot p$, is also continuous. The topological basis of $\beta S$ is $\{U_A: \emptyset \neq A \subseteq S \}$, where $U_A = \{p \in \beta S: A \in p \}$. The topological closure of a subset $X$ of $\beta S$ is denoted by $\overline{X}$. Then if $A \subseteq S$, it is easy to verifty that $\overline{A} = U_A$. It is known that $S$ is dense in $\beta S$. Given a compact right topological semigroup $(S, +)$, it has a smallest ideal $K(S)$, which is the union of all minimal left ideals of $S$ and also the union of all minimal right ideals of $S$. An idempotent $u \in S$ satisfies $u + u =u$; and if the idempotent $u \in K(S)$, $u$ is called minimal.

Let $(S, +)$ be a semigroup, $k \in \mathbb{N}$ and $\langle x_n \rangle_{n=1}^\infty$, $\langle y_n \rangle_{n=1}^\infty$ and $\langle x_n \rangle_{n=1}^k$ be three sequences in $S$. We say $\langle y_n \rangle_{n=1}^\infty$ is a sum subsystem of $\langle x_n \rangle_{n=1}^\infty$ if there exists a sequence $\langle H_n \rangle_{n=1}^\infty$ in $\mathcal{P}_f(\mathbb{N})$ such that for any $n \in \mathbb{N}$, $\max H_n < \min H_{n+1}$ and $y_n = \sum_{t \in H_n}x_t$. We have already defined $\mathrm{FS}(\langle x_n \rangle_{n=1}^\infty)$ in the above, the definition of $\mathrm{FS}(\langle x_n \rangle_{n=1}^k)$ is analogous to that of $\mathrm{FS}(\langle x_n \rangle_{n=1}^\infty)$.

Let $(S, +)$ be a semigroup. If $\mathcal{A}$ is a family of subsets of $S$, we say $\mathcal{A}$ satisfies partition regularity whenever for any $A \in \mathcal{A}$ and any 2-partition of $A$, there must exist one cell belonging to $\mathcal{A}$. A subset $A$ of $S$ is called a left solution set of $S$ (respectively, a right solution set of $S$) if there are $a, b \in S$ such that $A = \{x \in S: a + x = b \}$ (respectively, $A = \{x \in S: x + a = b \}$). Let $S$ be an infinite semigroup with size $\kappa$. We say $S$ is very weakly left cancellative (respectively, very weakly right cancellative) if the union of less than $\kappa$ left solution sets of $S$ (respectively, right solution sets of $S$) has size less than $\kappa$. We say $S$ is very weakly cancellative if it is both very weakly left cancellative and very weakly right cancellative. We say $S$ is weakly left cancellative (respectively, weakly right cancellative) if every left solution set (respectively, right solution set) is finite. And we say $S$ is left cancellative (respectively, right cancellative) if every left solution set (respectively, right solution set) has size $\leq 1$.

The definitions (or equivalent characterizations) of central sets, $J$-sets, $C$-sets and piecewise syndetic sets see \cite[Definition 4.42]{2012Algebra}, \cite[Definition 14.8.1, Definition 14.14.1(b), Theorem 14.14.7]{2012Algebra}, \cite[Definition 14.8.5, Definition 14.14.1 (d), Theorem 14.15.1]{2012Algebra} and \cite[Definition 4.38, Theorem 4.40]{2012Algebra}, respectively.

\section{IP sets}
In this section, we assume that semigroups have no idempotent, this assumption guarantees that every IP set is infinite, since we do not want to deal with IP sets of size 1 which are trivial. Observe that in this situation, for any IP set $A$ in a semigroup $(S, +)$, there exists an injective sequence $\langle x_n \rangle_{n=1}^\infty$ in $S$ such that $\mathrm{FS}(\langle x_n \rangle_{n=1}^\infty) \subseteq A$.

Then we consider the partition and almost disjoint problems of IP sets. Observe that IP sets can be countable and uncountable if the semigroups is uncountable, so we need to discuss these problems from two cases. First let us focus on the situation that the cardinality of IP sets is unknown, that is, find the partition and almost disjoint properties which are satisfied by any IP set. Based on this question, we obtain the following two results (Theorem \ref{IPctbl} and Theorem \ref{IPctblad}), which actually can be proved by minor modifications of the proof of \cite[Theorem 2.3]{2024Partition}. But here we provide a combinatorial argument, respectively.
\begin{theorem}\label{IPctbl}
  Suppose $(S, +)$ is an infinite left weakly cancellative semigroup and $A$ is an IP set in $S$. Then $A$ can be split into $\omega$ IP subsets.
\end{theorem}
\begin{proof}
  It is enough to show that $A$ can be split into two IP subsets. Since $A$ is IP and the assumption that $S$ has no idempotent, we can take an injective sequence $\langle a_n \rangle_{n=1}^\infty$ such that $\mathrm{FS}(\langle a_n \rangle_{n=1}^\infty) \subseteq A$. Let $B = \{ a_n: n \in \mathbb{N} \}$. Take $x_1 \in B$ and $y_1 \in B \setminus \{ x_1 \}$. Assume $k \in \mathbb{N}$ and we have obtained $\langle x_n \rangle_{n=1}^k$ and $\langle y_n \rangle_{n=1}^k$ such that $\mathrm{FS}(\langle x_n \rangle_{n=1}^k) \cap \mathrm{FS}(\langle y_n \rangle_{n=1}^k) = \emptyset$ and $\mathrm{FS}(\langle x_n \rangle_{n=1}^k) \cup \mathrm{FS}(\langle y_n \rangle_{n=1}^k) \subseteq A$. Then let $T_1 = \{ x \in S: \exists z_1 \in \mathrm{FS}(\langle y_n \rangle_{n=1}^k) \exists z_2 \in \mathrm{FS}(\langle x_n \rangle_{n=1}^k) (z_1 = z_2 + x) \}$. Since $S$ is left weakly cancellative, $T_1$ is a finite set. Then take $x_{k+1} \in B \setminus (\mathrm{FS}(\langle x_n \rangle_{n=1}^k) \cup \mathrm{FS}(\langle y_n \rangle_{n=1}^k) \cup T_1)$. Similarly, let $T_2 = \{ y \in S: \exists z_1 \in \mathrm{FS}(\langle x_n \rangle_{n=1}^{k+1}) \exists z_2 \in \mathrm{FS}(\langle y_n \rangle_{n=1}^k) (z_1 = z_2 + y) \}$ so $T_2$ is finite. Take $y_{k+1} \in B \setminus (\mathrm{FS}(\langle x_n \rangle_{n=1}^{k+1}) \cup \mathrm{FS}(\langle y_n \rangle_{n=1}^k) \cup T_2)$. 
  
  Obviously $\mathrm{FS}(\langle x_n \rangle_{n=1}^{k+1}) \cup \mathrm{FS}(\langle y_n \rangle_{n=1}^{k+1}) \subseteq A$. Assume that there is $z \in \mathrm{FS}(\langle x_n \rangle_{n=1}^{k+1}) \cap \mathrm{FS}(\langle y_n \rangle_{n=1}^{k+1})$. If $z = z^\prime + y_{k+1}$ for some $z^\prime \in \mathrm{FS}(\langle y_n \rangle_{n=1}^k$, then $y_{k+1} \in T_2$, contradiction. So either $z = y_{k+1}$ or $z \in \mathrm{FS}(\langle y_n \rangle_{n=1}^k$. If the former holds, according to the choice of $y_{k+1}$ we have $z \notin \mathrm{FS}(\langle x_n \rangle_{n=1}^{k+1})$, contradiction. Hence $z \in \mathrm{FS}(\langle y_n \rangle_{n=1}^k$. Since $z \in \mathrm{FS}(\langle x_n \rangle_{n=1}^{k+1})$, if $z \in \mathrm{FS}(\langle x_n \rangle_{n=1}^k)$, then $\mathrm{FS}(\langle x_n \rangle_{n=1}^k) \cap \mathrm{FS}(\langle y_n \rangle_{n=1}^k) \neq \emptyset$, contradicting to the inductive hypothesis; if $z = x_{k+1}$, then according to the choice of $x_{k+1}$ we have $z \notin \mathrm{FS}(\langle y_n \rangle_{n=1}^k$, contradiction. So there is only one case: $z = z^\prime + x_{k+1}$ for some $z^\prime \in \mathrm{FS}(\langle x_n \rangle_{n=1}^k$. But this case deduces that $x_{k+1} \in T_1$, which is also a contradiction. Therefore, $\mathrm{FS}(\langle x_n \rangle_{n=1}^{k+1}) \cap \mathrm{FS}(\langle y_n \rangle_{n=1}^{k+1}) = \emptyset$.
  
  Finally, we have $A_1 = \mathrm{FS}(\langle x_n \rangle_{n=1}^\infty)$ and $A_2 = \mathrm{FS}(\langle y_n \rangle_{n=1}^\infty)$. It is easy to see that $A_1 \cup A_2 \subseteq A$ and $A_1 \cap A_2 = \emptyset$.
\end{proof}

Before showing the almost disjoint result, we need to introduce a kind of special sequence.
\begin{definition}
  Suppose $(S, +)$ is a semigroup and $\langle x_n \rangle_{n=1}^\infty$ (respectively, $\langle x_n \rangle_{n=1}^k$ for some $k \in \mathbb{N}$) is a sequence in $S$. We say $\langle x_n \rangle_{n=1}^\infty$ (respectively, $\langle x_n \rangle_{n=1}^k$) satisfies finiteness of finite sums if for any nonempty $H_1, H_2 \in \mathcal{P}_f(\mathbb{N})$ (respectively, $H_1, H_2 \subseteq \{ 1, \ldots, k \}$), whenever $\max H_1 \neq \max H_2$, one must have $\sum_{n \in H_1}x_n \neq \sum_{n \in H_2}x_n$.
\end{definition}
From the definition, we can see that if $\langle x_n \rangle_{n=1}^\infty$ satisfies finiteness of finite sums, then for each $z \in \mathrm{FS}(\langle x_n \rangle_{n=1}^\infty)$, there are only finitely many $H \in \mathcal{P}_f(\mathbb{N})$ such that $z = \sum_{n \in H}x_n$. This explains the origin of the name ``finiteness of finite sums''. Actually this notion is similar to uniqueness of finite sums of sequences \cite[Page 3]{2008Largeness}, which needs $H$ to be unique for each $z$. We have the following property with respect to finiteness of finite sums.
\begin{lemma}\label{FFS}
  Suppose $(S, +)$ is an infinite left weakly cancellative semigroup and $\langle x_n \rangle_{n=1}^\infty$ is a sequence in $S$. There exists a sum subsystem $\langle y_n \rangle_{n=1}^\infty$ of $\langle x_n \rangle_{n=1}^\infty$ satisfying finiteness of finite sums.
\end{lemma}
\begin{proof}
  Take $H_1 = \{ 1 \}$ and $y_1 = x_1$. Assume $k \in \mathbb{N}$ and we have obtained $\langle H_n \rangle_{n=1}^k$ in $\mathcal{P}_f(\mathbb{N})$ and $\langle y_n \rangle_{n=1}^k$ in $S$ such that for each $n \in \{ 1, \ldots, k \}$, $y_n = \sum_{t \in H_n}x_t$ and if $n < k$, $\max H_n < \min H_{n+1}$, and $\langle y_n \rangle_{n=1}^k$ satisfies finiteness of finite sums. Let $Y = \{ y \in S: \exists z_1, z_2 \in \mathrm{FS}(\langle y_n \rangle_{n=1}^k) (z_1 + y = z_2) \}$. Since $S$ is left weakly cancellative, $Y$ is finite. Let $M = \max H_k$. Then take $y_{k+1} \in \mathrm{FS}(\langle x_n \rangle_{n=M+1}^\infty) \setminus (Y \cup \mathrm{FS}(\langle y_n \rangle_{n=1}^k))$. So there is $H_{k+1} \in \mathcal{P}_f(\mathbb{N})$ such that $y_{k+1} = \sum_{t \in H_{k+1}}x_t$ and $\min H_{k+1} > \max H_k$.
  
  Let us verify that $\langle y_n \rangle_{n=1}^{k+1}$ satisfies finiteness of finite sums. Take $G_1, G_2 \subseteq \{ 1, \ldots, k+1 \}$ satisfying $\max G_1 < \max G_2$. If $\max G_2 < k+1$, then $\sum_{n \in G_1}y_n \neq \sum_{n \in G_2}y_n$ by inductive hypothesis; if $\max G_2 = k+1$ and $|G_2| > 1$, then $G_2 = G_3 \cup \{ k+1 \}$ for some $G_3 \subseteq \{ 1, \ldots, k \}$. Assume $\sum_{n \in G_1}y_n = \sum_{n \in G_2}y_n$, then $\sum_{n \in G_1}y_n = \sum_{n \in G_3}y_n + y_{k+1}$, it turns out that $y_{k+1} \in Y$, contradiction, so $\sum_{n \in G_1}y_n \neq \sum_{n \in G_2}y_n$; otherwise, $G_2 = \{ k+1 \}$, according to the choice of $y_{k+1}$ we have $\sum_{n \in G_2}y_n = y_{k+1} \neq \sum_{n \in G_1}y_n$. Therefore, $\langle y_n \rangle_{n=1}^{k+1}$ satisfies finiteness of finite sums.
  
  Finally, we obtain a sum subsystem $\langle y_n \rangle_{n=1}^\infty$ of $\langle x_n \rangle_{n=1}^\infty$ satisfying finiteness of finite sums.
\end{proof}

Now that we have this lemma, we can incidentally get the following result with respect to uniqueness of finite sums, although we will not apply it.
\begin{theorem}
  Suppose $(S, +)$ is an infinite left weakly cancellative and right cancelltive semigroup and $\langle x_n \rangle_{n=1}^\infty$ is a sequence in $S$. There exists a sum subsystem $\langle y_n \rangle_{n=1}^\infty$ of $\langle x_n \rangle_{n=1}^\infty$ satisfying uniqueness of finite sums.
\end{theorem}
\begin{proof}
  Build a sum subsystem $\langle y_n \rangle_{n=1}^\infty$ of $\langle x_n \rangle_{n=1}^\infty$ as the same way of the proof of Lemma \ref{FFS}. So $\langle y_n \rangle_{n=1}^\infty$ satisfies finiteness of finite sums. Let us verify that it also satisfies uniqueness of finite sums. 
  
  Assume there exist two distinct $T, G \in \mathcal{P}_f(\mathbb{N})$ such that $\sum_{n \in T}y_n = \sum_{n \in G}y_n$. Write $T = \{ t_1, \ldots, t_k \}$ and $G = \{ g_1, \ldots, g_l \}$ for some $k , l \in \mathbb{N}$ such that $t_i < t_j$ and $g_i < g_j$ whenever $i < j$. If $k = l$, then by definition of finiteness of finite sums, we have $t_k = g_k$. So by right cancellative law of $S$, we have $\sum_{i=1}^{k-1}y_{t_i} = \sum_{i=1}^{k-1}y_{g_i}$ if $k > 1$. By the same argument, we can obtain $t_{k-1} = g_{k-1}$ and $\sum_{i=1}^{k-2}y_{t_i} = \sum_{i=1}^{k-2}y_{g_i}$ if $k > 2$. After finite steps, we will get that $t_i = g_i$ for each $i \in \{ 1, \ldots, k \}$, contradiction. So $k \neq l$, $k < l$ says. Again by the same argument, we obtain $t_k = g_l$, $t_{k-1} = g_{l-1}, \ldots$ $t_1 = g_{l-k+1}$. Notice that $\sum_{i=1}^{k}y_{t_i} = \sum_{i=1}^{l-k}y_{g_i} + \sum_{i=l-k+1}^{l}y_{g_i}$, so we have $y = z + y$ where $y = \sum_{i=1}^{k}y_{t_i}$ and $z = \sum_{i=1}^{l-k}y_{g_i}$. Then $z + y = z + (z + y)$, by right cancellative law of $S$ we have $z = z + z$, which contradicts with the assumption that $S$ has no idempotent. Therefore, $\langle y_n \rangle_{n=1}^\infty$ of $\langle x_n \rangle_{n=1}^\infty$ satisfies uniqueness of finite sums. 
\end{proof}

Then we have the following main result:
\begin{theorem}\label{IPctblad}
  Suppose $(S, +)$ is an infinite left weakly cancellative semigroup and $A$ is an IP set in $S$. Then $A$ contains $2^\omega$ almost disjoint IP subsets.
\end{theorem}
\begin{proof}
  Since $A$ is an IP set, we take a sequence $\langle x_n \rangle_{n=1}^\infty$ such that $\mathrm{FS}(\langle x_n \rangle_{n=1}^\infty) \subseteq A$. By Lemma \ref{FFS}, take a sum subsystem $\langle y_n \rangle_{n=1}^\infty$ of $\langle x_n \rangle_{n=1}^\infty$ satisfying finiteness of finite sums. Notice that $\langle y_n \rangle_{n=1}^\infty$ is an injective sequence, so by \cite[Chapter II, Theorem 1.3]{1980Set} take a family $\{ B_\alpha: \alpha < 2^\omega \}$ of $2^\omega$ almost disjoint subsets of $\{ y_n: n \in \mathbb{N} \}$. Let $A_\alpha = \mathrm{FS}(B_\alpha)$ for each $\alpha < 2^\omega$. So for each $\alpha < 2^\omega$, $A_\alpha$ is an IP set in $S$, and since $B_\alpha \subseteq \{ y_n: n \in \mathbb{N} \}$, we have $A_\alpha \subseteq \mathrm{FS}(\langle y_n \rangle_{n=1}^\infty) \subseteq \mathrm{FS}(\langle x_n \rangle_{n=1}^\infty) \subseteq A$. For any $\alpha < \beta < 2^\omega$, if $z \in A_\alpha \cap A_\beta$, then there exist $H_1, H_2 \in \mathcal{P}_f(\mathbb{N})$ such that $z = \sum_{n \in H_1}y_n = \sum_{n \in H_2}y_n$ and $\{ y_n: n \in H_1 \} \subseteq B_\alpha$ and $\{ y_n: n \in H_2 \} \subseteq B_\beta$. Since $\langle y_n \rangle_{n=1}^\infty$ satisfies finiteness of finite sums, $\max H_1 = \max H_2$, we denote $h = \max H_1$. So $y_h \in B_\alpha \cap B_\beta$. While $|B_\alpha \cap B_\beta| < \omega$, so there are only finite possible values of $h$, which implies that such $H_1$ and $H_2$ are also finitely many. Hence $|A_\alpha \cap A_\beta| < \omega$.
  
  So $\{ A_\alpha: \alpha < 2^\omega \}$ is as desired.
\end{proof}

We can easily see that the conclusions of Theorem \ref{IPctbl} and Theorem \ref{IPctblad} are already optimal for countable IP sets. So it is natural to ask whether uncountable IP sets have better properties. For this question, first we have the following result with respect to uncountable partition.
\begin{theorem}\label{NS}
  Suppose $S$ is an uncountable semigroup. Every uncountable IP set can be split into uncountably many IP subsets if and only if there is no uncountable non-IP set in $S$.
\end{theorem}
\begin{proof}
  The sufficiency is obvious. Conversely, assume there is an uncountable non-IP set $A$ in $S$, then we pick a countable IP set $B$ in $S$ arbitrarily so $A \cup B$ is also an IP set. But any uncountable splitting of $A \cup B$ must have one part, $C$ says, which is contained in $A$, so $C$ is not an IP set, contradiction.
\end{proof}

Since we mainly work in very weakly cancellative semigroups, and it is easy to provide an example of uncountable very weakly cancellative semigroups containing uncountable non-IP sets, $(\mathbb{R}^+, +)$ is such one with $(1, 2)$ as an uncountable non-IP subset. So one may further ask whether every uncountable very weakly cancellative semigroup has an uncountable non-IP subset? Here we give a counter-example.
\begin{theorem}\label{EX}
  There exists an uncountable very weakly cancellative semigroup with no idempotent, whose uncountable subsets are all IP sets.
\end{theorem}
\begin{proof}
  For each $\alpha < \omega_1$, take a set $A_\alpha$ of size $\omega$ such that $A_\alpha \cap A_\beta = \emptyset$ for each $\alpha < \beta < \omega_1$. We write $A_\alpha = \{ \alpha_n: n \in \mathbb{N} \}$ for each $\alpha < \omega_1$ and let $S = \bigcup_{\alpha < \omega_1}A_\alpha$. Then we define an operation $\oplus$ on $S$ by setting, for each $\alpha_n, \beta_m \in S$, $$\alpha_n \oplus \beta_m = \left\{ 
  \begin{array}{lcl}
  \beta_m, &  & {\alpha < \beta;}\\
  \alpha_n, &  & {\beta < \alpha;}\\
  \alpha_{n+m}, &  & {\alpha = \beta.}
  \end{array} \right.$$ 
  It is easy to see that $(S, \oplus)$ satisfies commutative law and has no idempotent. Now let us verify that it also satisfies associative law. For any $\alpha_{k_\alpha}, \beta_{k_\beta}, \gamma_{k_\gamma} \in S$, let $x = (\alpha_{k_\alpha} \oplus \beta_{k_\beta}) \oplus \gamma_{k_\gamma}$ and $y = \alpha_{k_\alpha} \oplus (\beta_{k_\beta} \oplus \gamma_{k_\gamma})$ for convenience. Then let us conduct classified discussion. If $\alpha, \beta$ and $\gamma$ are three distinct points of $\omega_1$, then both $x$ and $y$ are equal to $\delta_{k_\delta}$ where $\delta = \max\{\alpha, \beta, \gamma\}$; if $\alpha = \beta = \gamma$, then $x = y = \alpha_{k_\alpha + k_\beta + k_\gamma}$; if $\alpha = \beta < \gamma$, then $x = y = \gamma_{k_\gamma}$; if $\alpha = \beta > \gamma$, then $x = y = \alpha_{k_\alpha + k_\beta}$; By commutative law, the cases $\alpha < \beta = \gamma$ and $\alpha > \beta = \gamma$ are the same as the cases $\alpha = \beta > \gamma$ and $\alpha = \beta < \gamma$, respectively; if $\alpha = \gamma < \beta$, then $x = y = \beta_{k_\beta}$; otherwise, $\alpha = \gamma > \beta$, then $x = y = \alpha_{k_\alpha + k_\gamma}$. In conclusion, $x = y$ holds so $(S, \oplus)$ is an uncountable commutative semigroup. 
  
  Next we shall verify that $(S, \oplus)$ is very weakly cancellative. Take $\alpha_n, \beta_m \in S$ arbitrarily and let $B = \{x \in S: \alpha_n \oplus x = \beta_m \}$, so $B$ is a left solution set and, by commutative law, a right solution set. If $\alpha < \beta$, then $B = \{ \beta_m \}$; if $\alpha > \beta$, then $B = \emptyset$; if $\alpha = \beta$ and $n < m$, then $B = \{ \alpha_{m-n} \}$; if $\alpha = \beta$ and $n > m$, then $B = \emptyset$; otherwise, $\alpha = \beta$ and $n = m$, then $B = \{ \gamma_k \in S: \gamma < \alpha$ and $k \in \mathbb{N} \}$ so $|B| < \omega_1$. In conclusion, all left and right solution sets are countable so $(S, \oplus)$ is very weakly cancellative.
  
  Now let us show that every uncountable subset of $S$ is an IP set. Take an uncountable subset $A \subseteq S$. For each $\alpha < \omega_1$, let $B_\alpha = A \cap A_\alpha$ so $B_\alpha$ is countable. Note that $B_\alpha \cap B_\beta = \emptyset$ for any $\alpha < \beta < \omega_1$ and $A = \bigcup_{\alpha < \omega_1}B_\alpha$. Hence there exist uncountably many $B_\alpha$'s which are nonempty, then we pick one point from each of them to form an uncountable subset of $A$, which is an IP set, so $A$ is also an IP set. 
\end{proof}

At the end of this section, we give a result with respect to the almost disjoint problem when the IP set is known to be uncountable. 
\begin{theorem}\label{IPADU}
  Suppose $(S, +)$ is a semigroup of size $\kappa > \omega$, $\omega < \mu \leq \kappa$, $\mu$ contains $\lambda$ almost disjoint subsets and $A$ is an IP set in $S$ of size $\mu$. Then $A$ contains $\lambda$ almost disjoint IP subsets.
\end{theorem}
\begin{proof}
  Since $A$ is an IP set, there is a sequence $\langle a_n \rangle_{n=1}^\infty$ such that $\mathrm{FS}(\langle a_n \rangle_{n=1}^\infty) \subseteq A$. Take a family $\langle A_\alpha \rangle_{\alpha < \lambda}$ of $\lambda$ almost disjoint subsets of $A$ and let $B_\alpha = A_\alpha \cup \mathrm{FS}(\langle a_n \rangle_{n=1}^\infty)$ for each $\alpha < \lambda$. Hence each $B_\alpha$ is an IP subset of $A$ and has size $\mu$, and for any $\alpha < \beta < \lambda$, $|B_\alpha \cap B_\beta| \leq |A_\alpha \cap A_\beta \cup \mathrm{FS}(\langle a_n \rangle_{n=1}^\infty)| < \mu$. So $\langle B_\alpha \rangle_{\alpha < \lambda}$ is as desired.
\end{proof}

\section{Combinatorially rich sets}
In this section, we focus on another notion - combinatorially rich sets in $(\mathbb{N}, +)$. If $f \in L \subseteq {^n\mathbb{N}}$ and $a \in \mathbb{Z}$, then we denote $f + a$ to be the function $f(x) + a$ and $L + a = \{ f + a: f \in L \}$. To show the main result, we need the following lemma.

\begin{lemma}\label{lem1}
  Let $\langle r_n \rangle_{n=1}^\infty$ is a sequence in $\mathbb{N}$. 
  \begin{enumerate}
    \item[(i)] $\bigcup_{n=1}^\infty \mathcal{P}_n({^{r_n}\mathbb{N}})$ contains an almost disjoint family $\langle B_\alpha \rangle_{\alpha < 2^\omega}$ of size $2^\omega$ such that for each $\alpha < 2^\omega$ and $L \in \bigcup_{n=1}^\infty \mathcal{P}_n({^{r_n}\mathbb{N}})$, there is some $L^\prime \in B_\alpha$ such that $L^\prime = L + a$ for some $a \in \mathbb{N}$.
    \item[(ii)] $\bigcup_{n=1}^\infty \mathcal{P}_n({^{r_n}\mathbb{N}})$ contains a family $\langle B_\alpha \rangle_{\alpha < \omega}$ of pairwise disjoint infinite subsets such that for each $\alpha < \omega$ and $L \in \bigcup_{n=1}^\infty \mathcal{P}_n({^{r_n}\mathbb{N}})$, there is some $L^\prime \in B_\alpha$ such that $L^\prime = L + a$ for some $a \in \mathbb{N}$.
  \end{enumerate}

\end{lemma}
\begin{proof}
  The proofs of two items are essentially the same so here we show the first item. Let us define an equivalence relation $\sim$ on $\bigcup_{n=1}^\infty \mathcal{P}_n({^{r_n}\mathbb{N}})$: for each $L, L^\prime \in \bigcup_{n=1}^\infty \mathcal{P}_n({^{r_n}\mathbb{N}})$, $L \sim L^\prime$ if and only if there is some $a \in \mathbb{Z}$ such that $L = L^\prime + a$. Observe that there are $\omega$ equivalence classes, since any two distinct members $L, L^\prime$ are not equivalent if their sizes are not equal. Then we enumerate all equivalence classes as $\langle [L]_k \rangle_{k < \omega}$ and let $R = \{ L_k^\star: k < \omega$ and $L_k^\star \in [L]_k \}$ be the set of represent elements of equivalence classes. Let $\langle A_\alpha \rangle_{\alpha < 2^\omega}$ be an almost disjoint family of $\mathbb{N}$ and write $A_\alpha = \langle a_m^\alpha \rangle_{m < \omega}$ for each $\alpha < 2^\omega$. Then let $B_\alpha = \{L_k^\star + a_m^\alpha: k, m < \omega \}$ for each $\alpha < 2^\omega$. Let us verify that $\langle B_\alpha \rangle_{\alpha < 2^\omega}$ is as desired. 
  
  First, for each $\alpha < 2^\omega$, observe that for any distinct $k_1, k_2 < \omega$, $L_{k_1}^\star + a_{m_1}^\alpha \neq L_{k_2}^\star + a_{m_2}^\alpha$ for any $m_1, m_2 < \omega$; otherwise, $L_{k_1}^\star \sim L_{k_2}^\star$, which is a contradiction. Hence $|B_\alpha| = \omega$. Second, notice that for any $\alpha < \beta < \delta$, $B_\alpha \cap B_\beta = \{L \in \bigcup_{n=1}^\infty \mathcal{P}_n({^{r_n}S}): L = L_k^\star + a_{m_1}^\alpha = L_k^\star + a_{m_2}^\beta$ for some $k, m_1, m_2 < \omega \}$. This implies $|B_\alpha \cap B_\beta| \leq |A_\alpha \cap A_\beta| < \omega$. Third, for each $\alpha < 2^\omega$ and $L \in \bigcup_{n=1}^\infty \mathcal{P}_n({^{r_n}S})$, we pick some $k < \omega$ such that $L \sim L_k^\star$, so there is some $a \in \mathbb{Z}$ such that $L_k^\star = L + a$. We pick some $m < \omega$ such that $a_m^\alpha$ is larger than the absolute value of $a$ and let $L^\prime = L_k^\star + a_m^\alpha$. So $L^\prime \in B_\alpha$ and $L^\prime = L + b$ where $b = a + a_m^\alpha \in \mathbb{N}$.
\end{proof}

Then we have the following main result with respect to combinatorially rich sets.
\begin{theorem}\label{CR}
  Suppose $A$ is a combinatorially rich set in $(\mathbb{N}, +)$. Then
  \begin{enumerate}
    \item $A$ contains $2^\omega$ almost disjoint combinatorially rich subsets.
    \item $A$ can be split into $\omega$ pairwise disjoint combinatorially rich subsets.
  \end{enumerate}
\end{theorem}
\begin{proof}
    First we consider the first statement. Let $A$ be a combinatorially rich set in $\mathbb{N}$. By definition, for each $n \in \mathbb{N}$, we have $r_n \in \mathbb{N}$ such that for any $L \in \mathcal{P}_n({^{r_n}\mathbb{N}})$, there exist $a \in \mathbb{N}$ and nonempty $H \subseteq \{ 1, \ldots, r_n \}$ such that $S_L(a, H) \subseteq A$. Since the cardinality of $\bigcup_{n=1}^\infty \mathcal{P}_n({^{r_n}\mathbb{N}})$ is $\omega$, so we can enumerate it as $\langle L_k \rangle_{k < \omega}$. We will inductively build two $\omega$-sequences $\langle a_k \rangle_{k < \omega}$ and $\langle H_k \rangle_{k < \omega}$ such that for each $k < \omega$, $a_k \in \mathbb{N}$, $\emptyset \neq H_k \subseteq \{ 1, \ldots, r_{|L_k|} \}$ and $S_{L_k}(a_k, H_k) \subseteq A$, and for any $k < m < \omega$, $S_{L_k}(a_k, H_k) \cap S_{L_m}(a_m, H_m) = \emptyset$.
  
  For $L_0$, since $A$ is a combinatorially rich set, we obtain $a_0 \in S$ and nonempty $H_0 \subseteq \{ 1, \ldots, r_{|L_0|} \}$ such that $S_{L_0}(a_0, H_0) \subseteq A$. Let $0 < t < \omega$ and assume that $\langle a_k \rangle_{k < t}$ and $\langle H_k \rangle_{k < t}$ have been chosen. Let $m = |L_t|$ and $S_t = \bigcup_{k < t}S_{L_k}(a_k, H_k)$ so $|S_t| < \omega$.
  
  \begin{claim}
    Let $B = \{ (a, H) \in \mathbb{N} \times \mathcal{P}_f( \{ 1, 2, \ldots, r_m \}): S_{L_t}(a, H) \cap S_t \neq \emptyset \}$. Then $|B| < \omega$.
  \end{claim}
  \begin{proof}
    For each $H \in \mathcal{P}_f( \{ 1, 2, \ldots, r_m \})$, let $B_H = \{ a \in \mathbb{N}: S_{L_t}(a, H) \cap S_t \neq \emptyset \}$. Then $B = \bigcup \{ B_H \times \{ H \}: H \in \mathcal{P}_f( \{ 1, 2, \ldots, r_m \}) \}$. Since $\mathcal{P}_f( \{ 1, 2, \ldots, r_m \})$ is a finite set, it is enough to show that each $B_H$ is also finite. Observe that each $B_H = \bigcup_{x \in S_t}\bigcup_{f \in L_t} \{ a \in \mathbb{N}: x = a + \sum_{n \in H}f(n) \}$ and $|\{ a \in \mathbb{N}: x = a + \sum_{n \in H}f(n) \}| \leq 1$ when $x, f$ and $H$ are fixed, and $S_t$ and $L_t$ are also finite. So each $B_H$ is finite.
  \end{proof}
  \begin{claim}
    Let $C = \{ (a, H) \in \mathbb{N} \times \mathcal{P}_f( \{ 1, 2, \ldots, r_m \}): S_{L_t}(a, H) \subseteq A \}$. Then $|C| = \omega$.
  \end{claim}
  \begin{proof}
    Let $D = \{ a \in \mathbb{N}: \exists H \in \mathcal{P}_f( \{ 1, 2, \ldots, r_m \}) ( S_{L_t}(a, H) \subseteq A) \}$, so $|D| \leq |C|$. Hence it is enough to show that $|D| = \omega$. Since $A$ is a combinatorially rich set, $|D| \geq 1$. Assume $|D| \geq n$, $n \in \mathbb{N}$, take $a_1 < \ldots < a_n \in D$ and take $b \in \mathbb{N}$ larger than $a_n$. Since $L_t + b$ is also in $\mathcal{P}_m({^{r_m}\mathbb{N}})$, hence there exist $c \in \mathbb{N}$ and $H \in \mathcal{P}_f( \{ 1, 2, \ldots, r_m \})$ such that $S_{L_t + b}(c, H) \subseteq A$. Let $|H| = s$, then $S_{L_t + b}(c, H) = S_{L_t}(c + sb, H)$. So $c + sb \in D$ which is larger than $a_n$, it turns out that $|D| \geq n+1$. By induction, $|D| = \omega$.
  \end{proof}
  Then we take $(a_t, H_t) \in C \setminus B$. By induction, we obtain  $\langle a_k \rangle_{k < \omega}$ and $\langle H_k \rangle_{k < \omega}$ which is as desired. By Lemma \ref{lem1}(i), we pick an almost disjoint family $\langle B_\alpha \rangle_{\alpha < 2^\omega}$ of $\bigcup_{n=1}^\infty \mathcal{P}_n({^{r_n}\mathbb{N}})$, such that for each $\alpha < 2^\omega$ and $L \in \bigcup_{n=1}^\infty \mathcal{P}_n({^{r_n}\mathbb{N}})$, there is some $L^\prime \in B_\alpha$ such that $L^\prime = L + a$ for some $a \in \mathbb{N}$. Then for each $\alpha < 2^\omega$, define $A_\alpha = \bigcup\{ S_{L_k}(a_k, H_k): k < \omega$ and $L_k \in B_\alpha \}$. Let us verify that the family $\{A_\alpha: \alpha < 2^\omega \}$ is the witness of the first statement. Fix $\alpha < 2^\omega$. By construction we have $A_\alpha \subseteq A$; since $B_\alpha = \omega$ and for distinct $L_k, L_t \in B_\alpha$, $S_{L_k}(a_k, H_k) \cap S_{L_t}(a_t, H_t) = \emptyset$, so $|A_\alpha| = \omega$; for any $\beta < 2^\omega$ distinct from $\alpha$, $A_\alpha \cap A_\beta = \bigcup\{ S_{L_k}(a_k, H_k): k < \omega$ and $L_k \in B_\alpha \cap B_\beta \}$. Since $|B_\alpha \cap B_\beta| < \omega$, we have $|A_\alpha \cap A_\beta| < \omega$; for any $L \in \bigcup_{n=1}^\infty \mathcal{P}_n({^{r_n}\mathbb{N}})$, we pick some $k < \omega$ and $a \in \mathbb{N}$ such that $L_k \in B_\alpha$ and $L_k = L + a$. So $S_L(a_k + sa, H_k) = S_{L_k}(a_k, H_k) \subseteq A_\alpha$ where $s = |H_k|$, which means $A_\alpha$ is a combinatorially rich set.
  
  Therefore, $\{A_\alpha: \alpha < 2^\omega \}$ is an almost disjoint family of $A$, where each member is combinatorially rich. The proof of the second statement is essentially the same, using Lemma \ref{lem1}(ii) instead of Lemma \ref{lem1}(i).
\end{proof}
We do not know whether an analogous result of Theorem \ref{CR} holds in uncountable semigroups, so we close this section with this question.
\begin{question}
  If $(S, +)$ is an infinite semigroup of size $\kappa$ and $A$ is combinatorially rich in $S$, then does $A$ contain $\kappa$ pairwise disjoint combinatorially rich subsets? Moreover, if $\kappa$ contains $\lambda$ almost disjoint subsets, then does $A$ contain $\lambda$ almost disjoint combinatorially rich subsets? 
\end{question}

\section{An uncountable version of the polynomial extension of the central sets theorem}
In \cite[Section 2]{2023Polynomial}, authors introduced the notion of $J_p$-sets and $C_p$-sets and established a polynomial extension of the central sets theorem. However, we found that all these notions and relevant results in \cite{2023Polynomial} only focus on $(\mathbb{N}, +)$. Although authors noted that most of these results can be generalized to the case of countable commutative semigroups, the case of uncountable commutative semigroups is still unknown. In this section, we will establish an uncountable version of the polynomial extension of the central sets theorem.

Suppose $(S, +)$ is a commutative cancellative semigroup. We call $(S-S, +)$ is the difference group of $S$ where $S-S = \{ a - b: a, b \in S \}$ and $a - b$ is defined to be that element for which $(a - b) + b = a$. It is easy to see that it is an Abelian group. See \cite{2020A} for more information of difference groups. Moreover, if $(S-S, +, \cdot)$ is an integral domain and $j \in \mathbb{N}$, we say $f: (S-S)^j \rightarrow (S-S)$ is an integral polynomial on $(S-S)^j$ if it is a polynomial on $(S-S)^j$ with zero constant term and coefficients are in $S-S$. Let $\mathbb{P}_j$ be the set of integral polynomials on $(S-S)^j$ and let $\mathbb{P}$ denote $\mathbb{P}_1$. 

From now on we assume $(S, +)$ is a commutative cancellative semigroup without 0 and $(S-S, +, \cdot)$ is an integral domain. Then we observe that the following result holds, a version of \cite[Theorem 4.4]{2024Algebra}.

\begin{theorem}\label{ps}
  Suppose $j \in \mathbb{N}$, $u$ is an idempotent in $\beta(S^j)$, $R \in \mathcal{P}_f(\mathbb{P}_j)$, $A$ is a piecewise syndetic subset of $S-S$ and $L$ is a minimal left ideal of $\beta(S-S)$ such that $\overline{A} \cap L \neq \emptyset$. Then $\{\vec{x} \in S^j: \overline{A} \cap L \cap \bigcap_{f \in R} \overline{-f(\vec{x}) + A} \neq \emptyset \} \in u$.
\end{theorem}
\begin{proof}
  The proof is essentially the same as that of \cite[Theorem 4.4]{2024Algebra}.
\end{proof}

Then we have the following version of Abstract IP-Polynomial van der Waerden theorem for commutative cancellative semigroups.
\begin{corollary}\label{abs}
  Suppose $j \in \mathbb{N}$, $R \in \mathcal{P}_f(\mathbb{P}_j)$, $A$ is a piecewise syndetic subset of $S$ and $\langle \vec{y_n} \rangle_{n=1}^\infty$ is a sequence in $S^j$. Then there exist $a \in S$ and $H \in \mathcal{P}_f(\mathbb{N})$ such that for every $f \in R$, $a + f(\sum_{n \in H}\vec{y_n}) \in A$.
\end{corollary}
\begin{proof}
  By \cite[Lemma 5.11]{2012Algebra}, we pick an idempotent $u \in \bigcap_{m=1}^\infty \overline{\mathrm{FS}(\langle \vec{y_n} \rangle_{n=m}^\infty)}$. By \cite[Theorem 5]{2020A}, $A$ is also a piecewise syndetic subset of $S-S$, so there is a minimal left ideal $L$ of $\beta(S-S)$ such that $\overline{A} \cap L \neq \emptyset$. Then by Theorem \ref{ps}, $\{\vec{x} \in S^j: \overline{A} \cap L \cap \bigcap_{f \in R} \overline{-f(\vec{x}) + A} \neq \emptyset \} \in u$. In particular, $B \in u$ where $B = \{\vec{x} \in S^j: \overline{A} \cap \bigcap_{f \in R} \overline{-f(\vec{x}) + A} \neq \emptyset \}$. Since $\mathrm{FS}(\langle \vec{y_n} \rangle_{n=1}^\infty) \in u$, we have $B \cap \mathrm{FS}(\langle \vec{y_n} \rangle_{n=1}^\infty) \neq \emptyset$. Then pick some $H \in \mathcal{P}_f(\mathbb{N})$ such that $\sum_{n \in H}\vec{y_n} \in B$. Hence $\overline{A} \cap \bigcap_{f \in R} \overline{-f(\sum_{n \in H}\vec{y_n}) + A}$ is a nonempty open subset of $\beta (S-S)$. It is known that $S-S$ is dense in $\beta (S-S)$, so we can pick $a \in (S-S) \cap \overline{A} \cap \bigcap_{f \in R} \overline{-f(\sum_{n \in H}\vec{y_n}) + A}$. Observe that $a \in A \subseteq S$, so $a$ and $H$ are as desired.
\end{proof}

By the above corollary, we have the following result, where $S_{R, L}(a, H) = \{ a + f(\sum_{t \in H}g(t)): f \in R$ and $g \in L \}$.
\begin{theorem}\label{psj}
  Suppose, $m \in \mathbb{N}$, $A$ is a piecewise syndetic subset of $S$, $R \in \mathcal{P}_f(\mathbb{P})$ and $L \in \mathcal{P}_f({^{\mathbb{N}}S})$. Then there exist $a \in S$ and $H \in \mathcal{P}_f(\mathbb{N})$ such that $\min H > m$ and $S_{R, L}(a, H) \subseteq A$.
\end{theorem}
\begin{proof}
  The argument is the same as that of \cite[Theorem 2]{2023Polynomial} by applying Corollary \ref{abs}.
\end{proof}

Then we can introduce the more general definitions of $J_p$-sets and $C_p$-sets and then establish the uncountable version of the polynomial extension of the central set theorem as the way of \cite{2023Polynomial}.
\begin{definition}\label{defJC}
 Let $A \subseteq S$. 
\begin{enumerate}
  \item $A$ a $J_p$-set in $S$ if for every $R \in \mathcal{P}_f(\mathbb{P})$ and $L \in \mathcal{P}_f({^{\mathbb{N}}S})$, there exist $a \in S$ and $H \in \mathcal{P}_f(\mathbb{N})$ such that $S_{R, L}(a, H) \subseteq A$.
  \item $\mathcal{J}_p = \{ p \in \beta S: \forall X \in p (X$ is a $J_p$-set in $S) \}$.
  \item $A$ is a $C_p$-set in $S$ if there is an idempotent $p \in \mathcal{J}_p$ such that $A \in p$.
\end{enumerate}
\end{definition}

It is easy to see that every $J_p$-set in $S$ is a $J$-set. And notice that by Theorem \ref{psj}, every piecewise syndetic set in $S$ is a $J_p$-set, hence $K(\beta S) \subseteq \mathcal{J}_p$. It turns out that every central set in $S$ is a $C_p$-set. Also observe that all $C_p$-sets are $C$-sets. Now the following polynomial extension of the central set theorem for commutative cancellative semigroups is established, which also holds for central sets.
\begin{theorem}\label{UPV}
  Suppose $A$ is a $C_p$-set in $S$ and $R \in \mathcal{P}_f(\mathbb{P})$. Then there exist functions $\alpha: \mathcal{P}_f({^{\mathbb{N}}S})\rightarrow S$ and $H: \mathcal{P}_f({^{\mathbb{N}}S}) \rightarrow \mathcal{P}_f(\mathbb{N})$ such that 
  \begin{enumerate}
    \item if $L_1, L_2 \in \mathcal{P}_f({^{\mathbb{N}}S})$ and $L_1 \subsetneq L_2$, then $\max H(L_1) < \min H(L_2)$, and
    \item if $m \in \mathbb{N}$, $L_1, \ldots, L_m \in \mathcal{P}_f({^{\mathbb{N}}S})$, $L_1 \subsetneq \ldots \subsetneq L_m$ and $g_i \in L_i$ for each $i \in \{ 1, \ldots, m \}$, then for every $f \in R$, $\sum_{i=1}^{m}\alpha(L_i) + f\left(\sum_{i=1}^{m}\sum_{t \in H(L_i)}g_i(t)\right) \in A$.
  \end{enumerate}
\end{theorem}
\begin{proof}
  The proof is essentially the same as that of \cite[Theorem 11]{2023Polynomial}.
\end{proof}

\section{$C_p$-sets and PP-rich sets}

In this section, we continue to investigate the partition and almost disjoint properties of combinatorial notions. First of all, we have the following lemma.
\begin{lemma}\label{ideal}
  $\mathcal{J}_p$ is a compact ideal of $(\beta S, +)$.
\end{lemma}
\begin{proof}
  Trivially $\mathcal{J}_p$ is closed in $\beta S$ so it is compact.
  
  Let $p \in \mathcal{J}_p$, $q \in \beta S$, $R \in \mathcal{P}_f(\mathbb{P})$ and $L \in \mathcal{P}_f({^{\mathbb{N}}S})$. We shall show that $p + q, q + p \in \mathcal{J}_p$. To see that $q + p \in \mathcal{J}_p$, let $A \in q + p$. Then $\{ x \in S: -x + A \in p \} \in q$. We pick $x \in S$ satisfying $-x + A \in p$. Since $p \in \mathcal{J}_p$, we have $-x + A$ is a $J_p$-set. So pick $a \in S$ and $H \in \mathcal{P}_f(\mathbb{N})$ such that $S_{R, L}(a, H) \subseteq -x + A$. Hence $S_{R, L}(x + a, H) \subseteq A$, which implies that $A$ is a $J_p$-set so $q + p \in \mathcal{J}_p$.
  
  To see that $p + q \in \mathcal{J}_p$, let $A \in p + q$ and $B = \{ x \in S: -x + A \in q \}$. So $B \in p$. Then pick $a \in S$ and $H \in \mathcal{P}_f(\mathbb{N})$ such that $S_{R, L}(a, H) \subseteq B$. It turns out that $\bigcap_{f \in R}\bigcap_{g \in L}\left(-(a + f(\sum_{t \in H}g(t))) + A\right) \in q$. Then we pick a point $b$ from that intersection. It is easy to see that $S_{R, L}(a + b, H) \subseteq A$, so $A$ is a $J_p$-set and so $p + q \in \mathcal{J}_p$. 
\end{proof}

We remind the reader that $p \in \beta S$ is called uniform if for any $X \in p$, $|X| = |S|$. If $I$ is an ideal of $\beta S$ and $A \subseteq S$, then we call $A$ a uniform $I$-large subset of $S$ \cite[Page 3]{2024Partition} if there is a uniform idempotent $p \in I\cap \overline{A}$. Then we immediately obtain the following result for $C_p$-sets.
\begin{theorem}\label{Cp}
  Suppose $\kappa$ is an infinite cardinal and $|S| = \kappa$. 
  \begin{enumerate}
    \item If there is a family of $\delta$ almost disjoint subsets of $\kappa$, then every $C_p$-set contains $\delta$ almost disjoint $C_p$-subsets.
    \item Every $C_p$-set can be split into $\kappa$ $C_p$-subsets.
  \end{enumerate}
\end{theorem}
\begin{proof}
   By definition, a set $A$ is a $C_p$-set in $S$ if and only if there exists an idempotent $p \in \mathcal{J}_p \cap \overline{A}$. Since every $J_p$-set is a $J$-set, so by \cite[Theorem 3.2]{2024Partition}, every $J_p$-set has size $\kappa$, which deduces that $p$ is uniform. Then by Lemma \ref{ideal}, we have $A$ is a $C_p$-set if and only if $A$ is uniform $\mathcal{J}_p$-large. Hence the result follows from \cite[Theorem 2.3]{2024Partition} directly.
\end{proof}

After getting the partition and almost disjoint properties of $C_p$-sets, it is natural to consider $J_p$-sets. Unfortunately, it is difficult to obtain analogous results for $J_p$-sets. And the partition regularity of $J_p$-sets is also hard to obtain (which is an open question \cite[Question 17]{2023Polynomial} for the case $S = \mathbb{N}$). But we still have the following result, a partial answer of \cite[Question 17]{2023Polynomial}. Here we adopt $\mathbb{P} = \mathbb{P}_{\mathbb{N} \cup \{ 0 \}}(\mathbb{N}, \mathbb{N})$ \cite[page 3]{2023Polynomial} as the definition of the set of integral polynomials, so that coefficients of every integral polynomial are non-negative. 
\begin{theorem}\label{partialA}
  Suppose $A$ is a $J_p$-set in $(\mathbb{N}, +)$ and $B \subseteq \mathbb{N}$ is finite. Then $A \setminus B$ is also a $J_p$ set.
\end{theorem}
\begin{proof}
  Assume that $A \setminus B$ is not a $J_p$-set, then there exist $R \in \mathcal{P}_f(\mathbb{P})$ and $L \in \mathcal{P}_f({^{\mathbb{N}}\mathbb{N}})$ such that for any $a \in \mathbb{N}$ and $H \in \mathcal{P}_f(\mathbb{N})$, $S_{R, L}(a, H) \nsubseteq A \setminus B$. Write $L = \langle\langle x_{i, n} \rangle_{n=1}^\infty\rangle_{i=1}^l$ for some $l \in \mathbb{N}$. Now let us build an increasing sum subsystem of $L$. Let $K_1 = \{ 1 \}$. If $m \in \mathbb{N}$ and we have obtained $\langle K_j \rangle_{j=1}^m$ such that for each $j \in \{ 1, \ldots, m \}$, $K_j \in \mathcal{P}_f(\mathbb{N})$, and for each $i \in \{ 1, \ldots, l \}$ and each $j_1 < j_2 \leq m$, $\sum_{n \in K_{j_1}}x_{i, n} < \sum_{n \in K_{j_2}}x_{i, n}$ and $\max K_{j_1} < \min K_{j_2}$. Then pick $K_{m+1} \in \mathcal{P}_f(\mathbb{N})$ large enough such that $\max K_m < \min K_{m+1}$ and for each $i \in \{ 1, \ldots, l \}$, $\sum_{n \in K_m}x_{i, n} < \sum_{n \in K_{m+1}}x_{i, n}$. 
  
  Then for $i \in \{ 1, \ldots, l \}$ and $n \in \mathbb{N}$, let $y_{i, n} = \sum_{t \in K_n}x_{i, t}$. We denote $L^\prime = \langle\langle y_{i, n} \rangle_{n=1}^\infty \rangle_{i=1}^l$. By construction, each $\langle y_{i, n} \rangle_{n=1}^\infty$ is an increasing sequence, and it is a sum subsystem of $\langle x_{i, n} \rangle_{n=1}^\infty$. So $R$ and $L^\prime$ are also witnesses of the hypothesis.
  
  Meanwhile $A$ is a $J_p$-set, by \cite[Lemma 10]{2023Polynomial} for any $n \in \mathbb{N}$, there exist $a_n \in \mathbb{N}$ and $H_n \in \mathcal{P}_f(\mathbb{N})$ such that $\min H_n > n$ and $S_{R, L^\prime}(a_n, H_n) \subseteq A$. Then we take $\langle a_n \rangle_{n=1}^\infty$, $\langle H_n \rangle_{n=1}^\infty$, $f \in R$ and $g \in L^\prime$ such that for each $n \in \mathbb{N}$, $\min H_{n+1} > \max H_n$ and $a_n + f(\sum_{t \in H_n}g(t)) \in B$. Since the coefficients of $f$ are non-negative, $\{ a_n + f(\sum_{t \in H_n}g(t)): n \in \mathbb{N} \}$ is an infinite set, while $B$ is finite so a contradiction appears.
\end{proof}
  
There is another notion similar to $J_p$-sets: PP-rich sets. It was studied in \cite[Section 3]{2023Polynomial} as a family related to $J_p$-sets. However, authors still focus on $(\mathbb{N}, +)$ in \cite{2023Polynomial}. So here we extend this notion to commutative cancellative semigroups as the way of $J_p$-sets as follows.
\begin{definition}\label{PPR}
  If $A \subseteq S$, we say $A$ is a PP-rich subset of $S$ if for any $R \in \mathcal{P}_f(\mathbb{P})$, there exist $a, x \in S$ such that $S_R(a, x) \subseteq A$ where $S_R(a, x) = \{ a + f(x): f \in R \}$.
\end{definition}
It is easy to see that all $J_p$-sets are PP-rich sets. Actually PP-rich sets satisfy the following seemingly stronger assertion.
\begin{lemma}\label{PPlem}
  $A$ is PP-rich in $S$ if and only if for any $R \in \mathcal{P}_f(\mathbb{P})$, there exist $a \in A$ and $x \in S$ such that $S_R(a, x) \subseteq A$.
\end{lemma}
\begin{proof}
  The sufficiency is trivial.
  
  For necessity, let $R \in \mathcal{P}_f(\mathbb{P})$ and let $f \in R$. Then let $R^\prime = f + R = \{ f + h: h \in R \}$. Since $A$ is PP-rich, there exist $a, x \in S$ such that $S_{R^\prime \cup \{ f \}}(a, x) \subseteq A$. Let $b = a + f(x)$. Then $b \in A$ and $S_R(b ,x) \subseteq A$. 
\end{proof}

By applying the above lemma, we can determine the size of PP-rich sets under certain conditions. For convenience, denote $(S-S) \setminus \{ 0 \} = (S-S)_0$.
\begin{theorem}\label{PP}
  Suppose $\kappa$ is an infinite regular cardinal, $|S| = \kappa$ and $((S-S)_0, \cdot)$ is a very weakly cancellative semigroup. Then every PP-rich set in $S$ has size $\kappa$.
\end{theorem}
\begin{proof}
  Assume there is a PP-rich set $A$ of size $< \kappa$. Let $X = \{ x \in S-S: \exists c, d \in A (c + x = d) \}$ so $|X| < \kappa$. Pick $b \in (S-S)_0$. Since $((S-S)_0, \cdot)$ is very weakly cancellative, $y^{-1}X = \{ a \in S-S: ay \in X \}$ has size less than $\kappa$ for any $y \in (S-S)_0$. In particular, $|b^{-1}X| < \kappa$. Let $Y = \bigcup_{x \in b^{-1}X \setminus \{ 0 \}}x^{-1}X$. Since $\kappa$ is regular, $|Y| < \kappa$. Pick $a \in (S-S)_0 \setminus Y$. Then for each $x \in b^{-1}X \setminus \{ 0 \}$, $a \notin x^{-1}X$ so $ax \notin X$. 
  
  Now let $f(x) = ax$ and $h(x) = bx$, so both of them are integral polynomials. Then take $R = \{ f, h \}$. Since $A$ is a PP-rich set, by Lemma \ref{PPlem}, there exist $t \in A$ and $x \in S$ such that $t + ax, t + bx \in A$. Hence $ax, bx \in X$ so $x \in b^{-1}X \setminus \{ 0 \}$. By construction, we have $ax \notin X$, contradiction.
\end{proof}
From the proof, it is easy to observe that if $((S-S)_0, \cdot)$ is cancellative, then the conclusion still holds even if the cardinality of $S$ is singular.

In \cite[Theorem 19]{2023Polynomial}, authors proved the partition regularity of PP-rich sets for $S = \mathbb{N}$, that is, for any 2-partition of a PP-rich set, there must be a PP-rich cell. A natural question arises: Can any PP-rich set in $\mathbb{N}$ be split into $\omega$ PP-rich subsets? Based on this question, we did some further work and obtained the following result:
\begin{theorem}\label{PPN}
\begin{enumerate}
  \item Every PP-rich set in $\mathbb{N}$ contains $2^\omega$ almost disjoint PP-rich subsets.
  \item Every PP-rich set in $\mathbb{N}$ can be split into $\omega$ PP-rich subsets.
\end{enumerate}

\end{theorem}
\begin{proof}
  Let $A$ be a PP-rich set in $\mathbb{N}$. Enumerate $\mathcal{P}_f(\mathbb{P})$ as $\langle R_n \rangle_{ n < \omega }$. We will inductively build two $\omega$-sequences $\langle a_n \rangle_{ n < \omega }$ and $\langle x_n \rangle_{ n < \omega }$ in $\mathbb{N}$ such that for each $n < \omega$, $S_{R_n}(a_n , x_n) \subseteq A$ and for each $m < n < \omega$, $S_{R_m}(a_m , x_m) \cap S_{R_n}(a_n , x_n) = \emptyset$.
  
  Since $A$ is PP-rich, pick $a_0, x_0 \in \mathbb{N}$ such that $S_{R_0}(a_0 , x_0) \subseteq A$. Let $0 < k < \omega$ and assume that $\langle a_n \rangle_{ n < k }$ and $\langle x_n \rangle_{ n < k }$ have been chosen. Let $S_k = \bigcup_{n < k}S_{R_n}(a_n , x_n)$. Note that $S_k$ is finite. So by Theorem \ref{PP}, $S_k$ is not a PP-rich set. Then by \cite[Theorem 19]{2023Polynomial}, $A \setminus S_k$ is PP-rich. Pick $a_k, x_k \in \mathbb{N}$ such that $S_{R_k}(a_k , x_k) \subseteq A \setminus S_k$. It is easy to see that $a_k, x_k$ are as desired.
  
  By \cite[Lemma 2.1(i)]{2008Almost}, we obtain a family $\{ \mathcal{A}_\alpha: \alpha < 2^\omega \}$ of almost disjoint subsets of $\mathcal{P}_f(\mathbb{P})$, such that for each $\alpha < 2^\omega$ and each $R \in \mathcal{P}_f(\mathbb{P})$ there exists $G \in \mathcal{A}_\alpha$ such that $R \subseteq G$. Then for each $\alpha < 2^\omega$, let $B_\alpha = \bigcup \{ S_{R_n}(a_n , x_n): n < \omega$ and $R_n \in \mathcal{A}_\alpha \}$. Since each $\mathcal{A}_\alpha$ has size $\omega$, $|B_\alpha| = \omega$ for each $\alpha < 2^\omega$. And observe that for any $\alpha < \beta < 2^\omega$, $B_\alpha \cap B_\beta =\bigcup \{ S_{R_n}(a_n , x_n): n < \omega$ and $R_n \in \mathcal{A}_\alpha \cap \mathcal{A}_\beta \}$, so by the fact that $|\mathcal{A}_\alpha \cap \mathcal{A}_\beta| < \omega$ we have $|B_\alpha \cap B_\beta| < \omega$. Hence $\{ B_\alpha: \alpha < 2^\omega \}$ is a family of almost disjoint subsets of $A$. It is enough to show that each $B_\alpha$ is PP-rich. Fix $\alpha < 2^\omega$. For any $R \in \mathcal{P}_f(\mathbb{P})$, we pick $G \in \mathcal{A}_\alpha$ such that $R \subseteq G$. Note that $G = R_n$ for some $n \in \mathbb{N}$ so $S_R(a_n , x_n)  \subseteq S_{R_n}(a_n , x_n) \subseteq B_\alpha$.
  
  The proof of the second statement is essentially the same, using \cite[Lemma 2.1(ii)]{2008Almost} instead of \cite[Lemma 2.1(i)]{2008Almost}.
\end{proof}

When the semigroup is uncountable, we do not know whether PP-rich sets still have partition regularity. So we close this section with following questions.
\begin{question}
  Do PP-rich sets have partition regularity in uncountable semigroups?
\end{question}
\begin{question}
  Are PP-rich sets still have partition and almost disjoint properties when $S$ is uncountable?
\end{question} 
\begin{question}
  Do $J_p$-sets have any partition or almost disjoint properties? Or partition regularity?
\end{question}

\section*{Acknowledgements}
   I acknowledge support received from NSFC No. 12401002.

\bibliographystyle{plain}

\begin{thebibliography}{9}

\bibitem{2008Largeness}
C. Adams, N. Hindman and D. Strauss,
Largeness of the set of finite products in a semigroup, 
Semigroup Forum 76 (2008), 276-296.

\bibitem{2023Polynomial}
L. L. Baglini, S. Goswami and S. K. Patra,
Polynomial extension of the Stronger Central Sets Theorem,
The Electronic Journal of Combinatorics, 30(4) (2023), $\sharp$P4.36.

\bibitem{2003Ball}
T. Banach and I. Protasov, 
Ball Structures and Colorings of Graphs and Groups, 
Mathematical Studies Monograph Series, vol. 11, VNTL Publishers, L’viv, 2003.

\bibitem{2020On}
V. Bergelson and D. Glasscock,
On the interplay between additive and multiplicative largeness and its combinatorial applications,
Journal of Combinatorial Theory, Series A, 172(2020), 105203.

\bibitem{1990Nonmetrizable}
V. Bergelson, N. Hindman, 
Nonmetrizable topological dynamics and Ramsey theory, 
Transactions of the American Mathematical Society, 320 (1990), 293–320.

\bibitem{2023Combinatorially}
N. Hindman, H. Hosseini, D. Strauss and M. Tootkaboni,
Combinatorially rich sets in arbitrary semigroups,
Semigroup Forum, 107 (2023), 127-143.

\bibitem{2008Almost}
T. J. Carlson, N. Hindman, J. McLeod and D. Strauss.
Almost disjoint large subsets of semigroups.
Topology and its Applications. 155 (2008), 433-444.

\bibitem{2008A}
D. De, N. Hindman and D. Strauss,
A new and stronger Central Set Theorem,
Fundamenta Mathematicae, 199(2008), 155-175.

\bibitem{1981Recurrence}
H. Furstenberg, 
Recurrence in Ergodic Theory and Combinatorical Number Theory, 
Princeton University Press, Princeton, 1981.

\bibitem{2020A}
S. Goswami and S. Jana,
A combinatorial viewpoint on preserving notion of largeness and an abstract Rado theorem,
arXiv.2008.06501 (2020).

\bibitem{1974Finite}
N. Hindman, 
Finite sums from sequences within cells of a partition of $\mathbb{N}$,
Journal of Combinatorial Theory (Series A), 17 (1974), 1–11.

\bibitem{2003Infinite}
N. Hindman, I. Leader and D. Strauss, 
Infinite partition regular matrices—solutions in central sets, 
Transactions of the American Mathematical Society, 355 (2003), 1213–1235.

\bibitem{2024Algebra}
N. Hindman and D. Strauss,
Algebra in the Stone-\v{C}ech compactification - an update,
Topology Proceedings, 64 (2024), 1-69.

\bibitem{2012Algebra}
N. Hindman and D. Strauss,
Algebra in the Stone-\v{C}ech compactification: theory and applications,
2nd ed, de Gruyter, Berlin (2012).

\bibitem{1980Set}
K. Kunen,
Set theory: An introduction to independence proofs,
North-Holland
Publishing Co., Amsterdam (1980).

\bibitem{2011Partitions}
I. Protasov,
Partitions of groups into thin subsets, 
Algebra and Discrete Mathematics, 11(2) (2011), 78–81.

\bibitem{2015Partitions}
I. Protasov and S. Slobodianiuk,
Partitions of groups into large subsets,
Journal of Group Theory, 18(2) (2015), 291–298.

\bibitem{2012Prethick}
I. Protasov and S. Slobodianiuk, 
Prethick subsets in partitions of groups, 
Algebra and Discrete Mathematics, 14(2) (2012), 1–9.

\bibitem{1996Nonmetrizable}
H. Shi, H. Yang, 
Nonmetrizable topological dynamic characterization of central sets, 
Fundamenta Mathematics, 150 (1996) 1–9.

\bibitem{2024Partition}
T. Zhang, 
Partition of large subsets of semigroups,
Journal of Symbolic Logic, published online(2024) DOI: 10.1017/jsl.2023.102, 1-6.

\end{thebibliography}

\end{document}